\documentclass[12pt]{amsart}
\usepackage{graphics}
\usepackage{amsfonts,amssymb,color}
\usepackage[mathscr]{eucal}
\usepackage{amsmath, amsthm}
\usepackage{mathrsfs}
\usepackage{amsbsy}
\usepackage{dsfont}
\usepackage{bbm}
\usepackage{wasysym}
\usepackage{stmaryrd}
\usepackage{url}
\usepackage{fourier}
\usepackage[T1]{fontenc}

\input xypic
\xyoption{all}

\makeindex
\makeglossary

\begin{document}
\baselineskip = 16pt

\newcommand \ZZ {{\mathbb Z}}
\newcommand \NN {{\mathbb N}}
\newcommand \RR {{\mathbb R}}
\newcommand \PR {{\mathbb P}}
\newcommand \AF {{\mathbb A}}
\newcommand \GG {{\mathbb G}}
\newcommand \QQ {{\mathbb Q}}
\newcommand \bcA {{\mathscr A}}
\newcommand \bcC {{\mathscr C}}
\newcommand \bcD {{\mathscr D}}
\newcommand \bcF {{\mathscr F}}
\newcommand \bcG {{\mathscr G}}
\newcommand \bcH {{\mathscr H}}
\newcommand \bcM {{\mathscr M}}
\newcommand \bcJ {{\mathscr J}}
\newcommand \bcL {{\mathscr L}}
\newcommand \bcO {{\mathscr O}}
\newcommand \bcP {{\mathscr P}}
\newcommand \bcQ {{\mathscr Q}}
\newcommand \bcR {{\mathscr R}}
\newcommand \bcS {{\mathscr S}}
\newcommand \bcV {{\mathscr V}}
\newcommand \bcW {{\mathscr W}}
\newcommand \bcX {{\mathscr X}}
\newcommand \bcY {{\mathscr Y}}
\newcommand \bcZ {{\mathscr Z}}
\newcommand \goa {{\mathfrak a}}
\newcommand \gob {{\mathfrak b}}
\newcommand \goc {{\mathfrak c}}
\newcommand \gom {{\mathfrak m}}
\newcommand \gon {{\mathfrak n}}
\newcommand \gop {{\mathfrak p}}
\newcommand \goq {{\mathfrak q}}
\newcommand \goQ {{\mathfrak Q}}
\newcommand \goP {{\mathfrak P}}
\newcommand \goM {{\mathfrak M}}
\newcommand \goN {{\mathfrak N}}
\newcommand \uno {{\mathbbm 1}}
\newcommand \Le {{\mathbbm L}}
\newcommand \Spec {{\rm {Spec}}}
\newcommand \Gr {{\rm {Gr}}}
\newcommand \Pic {{\rm {Pic}}}
\newcommand \Jac {{{J}}}
\newcommand \Alb {{\rm {Alb}}}
\newcommand \Corr {{Corr}}
\newcommand \Chow {{\mathscr C}}
\newcommand \Sym {{\rm {Sym}}}
\newcommand \Prym {{\rm {Prym}}}
\newcommand \cha {{\rm {char}}}
\newcommand \eff {{\rm {eff}}}
\newcommand \tr {{\rm {tr}}}
\newcommand \Tr {{\rm {Tr}}}
\newcommand \pr {{\rm {pr}}}
\newcommand \ev {{\it {ev}}}
\newcommand \cl {{\rm {cl}}}
\newcommand \interior {{\rm {Int}}}
\newcommand \sep {{\rm {sep}}}
\newcommand \td {{\rm {tdeg}}}
\newcommand \alg {{\rm {alg}}}
\newcommand \im {{\rm im}}
\newcommand \gr {{\rm {gr}}}
\newcommand \op {{\rm op}}
\newcommand \Hom {{\rm Hom}}
\newcommand \Hilb {{\rm Hilb}}
\newcommand \Sch {{\mathscr S\! }{\it ch}}
\newcommand \cHilb {{\mathscr H\! }{\it ilb}}
\newcommand \cHom {{\mathscr H\! }{\it om}}
\newcommand \colim {{{\rm colim}\, }} 
\newcommand \End {{\rm {End}}}
\newcommand \coker {{\rm {coker}}}
\newcommand \id {{\rm {id}}}
\newcommand \van {{\rm {van}}}
\newcommand \spc {{\rm {sp}}}
\newcommand \Ob {{\rm Ob}}
\newcommand \Aut {{\rm Aut}}
\newcommand \cor {{\rm {cor}}}
\newcommand \Cor {{\it {Corr}}}
\newcommand \res {{\rm {res}}}
\newcommand \red {{\rm{red}}}
\newcommand \Gal {{\rm {Gal}}}
\newcommand \PGL {{\rm {PGL}}}
\newcommand \Bl {{\rm {Bl}}}
\newcommand \Sing {{\rm {Sing}}}
\newcommand \spn {{\rm {span}}}
\newcommand \Nm {{\rm {Nm}}}
\newcommand \inv {{\rm {inv}}}
\newcommand \codim {{\rm {codim}}}
\newcommand \Div{{\rm{Div}}}
\newcommand \sg {{\Sigma }}
\newcommand \DM {{\sf DM}}
\newcommand \Gm {{{\mathbb G}_{\rm m}}}
\newcommand \tame {\rm {tame }}
\newcommand \znak {{\natural }}
\newcommand \lra {\longrightarrow}
\newcommand \hra {\hookrightarrow}
\newcommand \rra {\rightrightarrows}
\newcommand \ord {{\rm {ord}}}
\newcommand \Rat {{\mathscr Rat}}
\newcommand \rd {{\rm {red}}}
\newcommand \bSpec {{\bf {Spec}}}
\newcommand \Proj {{\rm {Proj}}}
\newcommand \pdiv {{\rm {div}}}
\newcommand \CH {{\it {CH}}}
\newcommand \wt {\widetilde }
\newcommand \ac {\acute }
\newcommand \ch {\check }
\newcommand \ol {\overline }
\newcommand \Th {\Theta}
\newcommand \cAb {{\mathscr A\! }{\it b}}

\newenvironment{pf}{\par\noindent{\em Proof}.}{\hfill\framebox(6,6)
\par\medskip}

\newtheorem{theorem}[subsection]{Theorem}
\newtheorem{conjecture}[subsection]{Conjecture}
\newtheorem{proposition}[subsection]{Proposition}
\newtheorem{lemma}[subsection]{Lemma}
\newtheorem{remark}[subsection]{Remark}
\newtheorem{remarks}[subsection]{Remarks}
\newtheorem{definition}[subsection]{Definition}
\newtheorem{corollary}[subsection]{Corollary}
\newtheorem{example}[subsection]{Example}
\newtheorem{examples}[subsection]{examples}

\title{Blow ups and base changes of symmetric powers and Chow groups }
\author{Kalyan Banerjee}

\address{Tata Institute of Fundamental Research Mumbai, India}

\email{kalyan@math.tifr.res.in}

\footnotetext{Mathematics Classification Number: 14C25, 14D05, 14D20,
 14D21}

\begin{abstract}
Let $\Sym^m X$ denote the $m$-th symmetric power of  a smooth projective curve $X$. Let $\wt{\Sym^m X}$ be the blow up of $\Sym^m X$ along some non-singular subvariety. In this note we are going to discuss when the pushforward homomorphism induced by the natural morphism from $\wt{\Sym^m X}$ to $\Sym^n X$ is injective at the level of Chow groups for $m\leq n$. Also we are going to prove that base changes of embeddings of one symmetric power into another, with respect to embeddings induces an injection at the level of Chow groups.
\end{abstract}

\maketitle

\section{Introduction}
In the paper by \cite{Collino} the injectivity of the push-forward homomorphism at the level of Chow groups, induced by the closed embedding of one symmetric power of a smooth projective curve into another had been proved. In the paper \cite{BI}, the techniques developed in \cite{Collino} has been generalised to answer similar injectivity question at the level of higher Chow groups, induced by the closed embedding of one symmetric power of a curve into another. Then in \cite{BAN} a generalisation says that the same injectivity holds at the level of Chow groups for symmetric powers of higher dimensional varieties.

In this paper our aim is to address the following questions. Let us consider the closed embedding of one symmetric power of a  fixed projective curve into another higher dimensional symmetric power of the same projective curve. Say we have $\Sym^m X$ embedded inside $\Sym^n X$. Let $Z$ be a smooth subvariety of $\Sym^m X$ intersecting $\Sym^{m-1}X$ transversally. Then we blow up $\Sym^m X$ along $Z$, let $\wt{\Sym^m X}$ denote the blow up. Then we prove that the natural morphism from $\wt{\Sym^m X}$ to $\Sym^n X$ induces an injective push-forward homomorphism at the level of Chow groups of zero cycles with the assumption that $\CH_0(\wt{\Sym^{m-1}X})\to \CH_0(\Sym^n X)$ is injective, where $\wt{\Sym^{m-1}X}$ is the strict transform of $\Sym^{m-1}X$ under this blow-up.

\textit{Let $Z$ be a smooth subvariety intersecting  a copy of $\Sym^{m-1}X $ transversally in $\Sym^m X$. Let $\pi^{-1}(Z)=E$. Let $\wt{\Sym^{m-1}X}$ is the strict transform of $\Sym^{m-1}X$ under this blow-up. Also assume that $\CH_0(\wt{\Sym^{m-1}X})\to \CH_0(\Sym^n X)$ is injective.  Then $\pi_*$ from $\CH_0(\wt{\Sym^m X})$ to $\CH_0(\Sym^n X)$ is injective.}

The technique is to follow the Collino's technique in \cite{Collino} to prove the injectivity of the push-forward induced by the closed embedding $\Sym^m C$ into $\Sym^n C$, where $C$ is a smooth projective. There is a natural correspondence on $\Sym^n C\times \Sym^m C$ induced by the graph of the projection from $C^n\to C^m$. This correspondence reduces everything to chase a commutative diagram where the rows are the localisation exact sequences at the level of Chow groups, which gives the required result. It is worthwhile to mention that the method involved in the proof can also be performed in the set up when we consider the group of algebraic cycles modulo algebraic equivalence on symmetric powers.

In the next section we try to understand the Collino's technique for base change of symmetric powers of curves. We prove that if we have an open or closed immersion of algebraic schemes, then the embedding for base change of symmetric powers induces an injection at the level of Chow groups. In \cite{BI} we did the same for algebraic cycles modulo algebraic equivalence but still we present this theorem here to understand when we have base change with respect to morphisms which are not embeddings. Further we consider the projective varieties $Y$ and a group $G$ acting on $Y$, and consider the base change of symmetric powers with respect to the natural morphism from $Y$ to $Y/G$. We prove that the push-forward from $CH_*\left(\frac{\Sym^m X\times_{Y/G}Y}{G}\right)_{\QQ}$ to $\CH_*\left(\frac{\Sym^n X\times_{Y/G}Y}{G}\right)_{\QQ}$ is injective. The main result is as follows:

\textit{Let $G$ be a finite group acting on $Y$ and we have a fiber square. Suppose that we have a morphism from $\Sym^n X$ to $Y/G$. Suppose that $\Sym^i X\times_{Y/G}Y$ is smooth . Then the embedding $\Sym^m X\times_{Y/G}Y$ into $\Sym^n X\times_{Y/G}Y$ induces injective push-forward homomorphism from $\CH_*\left(\frac{\Sym^m X\times_{Y/G}Y}{G}\right)_{\QQ}$ into $\CH_*\left(\frac{\Sym^n X\times_{Y/G}Y}{G}\right)_{\QQ}$ (or we can resolve the singularities and consider the statement on the resolution of singularities).}

In the third section we understand the case when we have a branched cover of a smooth projective curve by another smooth projective curve. That is we ask what is the kernel of the push-forward homomorphism at the level of Chow groups of symmetric powers, induced by the finite morphism $C'\to C$ of smooth projective curves. Suppose that the covering is $i:1$ and the cyclic group $\ZZ_i$ acts on $C'$ (here $\ZZ_i$ is the finite group $\ZZ/i\ZZ$). Then the natural homomorphism from $\CH_*(\Sym^m C)^{\ZZ_i}$ to $\CH_*(\Sym^n C)$ has only torsion elements in the kernel . The main result is as follows.

\textit{Elements of the kernel of the push-forward from $\CH_*(\Sym^m C')^{\ZZ_i}$ to $\CH_*(\Sym^n C)$ are torsion.}

In the last section we consider a non-singular closed subscheme $E$ inside some $\Sym^m C$ ($C$ is a smooth projective curve) and prove that the closed embedding  $E\to \Sym^m C\to \Sym^n C$ induces an injective push-forward homomorphism at the level of Chow groups. In particular the push-forward $\CH_*(E)\to \CH_*(\Sym^m C)$ is injective.

\textit{Let $E$ be a non-singular closed subscheme in $\Sym^m C$, such that its intersections with $\Sym^l C$ for all $l<m$ inside $\Sym^m C$ are non-singular. Then the closed embedding of $E$ into $\Sym^n C$ for $m\leq n$ induces a push-forward homomorphism at the level of Chow groups which has torsion kernel.}

\medskip

{\small \textbf{Acknowledgements:}The author would like to thank the ISF-UGC project for funding this project and also thanks the hospitality of Indian Statistical Institute, Bangalore Center for hosting this project.

We assume that the ground field is algebraically closed and of characteristic zero.}

\section{Blow up of symmetric powers and Chow groups}
Let $X$ be a projective curve. Consider  a smooth subvariety $Z$ inside $\Sym^m X$ and blow up $\Sym^m X$ along $Z$. Call it $\wt{\Sym^m X}$ (it is smooth). We have natural morphism $\pi$ from $\wt{\Sym^m X}$ to $\Sym^n X$, given by the composition $\wt{\Sym^m X}\to \Sym^m X$ and $\Sym^m X\to \Sym^n X$, where the later morphism is
$$[x_1,\cdots,x_m]\mapsto [x_1,\cdots,x_m,p,\cdots,p]\;,$$
$p$ is a fixed point in $X$. We would like to investigate when $\pi_*$
is injective.
\begin{proposition}
\label{prop1}
Let $Z$ be a smooth subvariety intersecting  a copy of   $\Sym^{m-1}X $ transversally in $\Sym^m X$. Let $\pi^{-1}(Z)=E$. Let $\wt{\Sym^{m-1}X}$ is the strict transform of $\Sym^{m-1}X$ under this blow-up. Also assume that $\CH_0(\wt{\Sym^{m-1}X})\to \CH_0(\Sym^n X)$ is injective.  Then $\pi_*$ from $\CH_0(\wt{\Sym^m X})$ to $\CH_0(\Sym^n X)$ is injective.
\end{proposition}
\begin{proof}
Our approach is to follow that approach of Collino as presented in \cite{Collino} to prove that the closed embedding of $\Sym^m X$ into $\Sym^n X$ induces an injective push-forward homomorphism at the level of Chow groups. So first we consider the correspondence
$$\Gamma=\pi_n\times \pi_m(Graph(pr_{n,m}))$$
where $pr_{n,m}$ is the projection morphism from $X^n$ to $X^m$ and $\pi_i$ is the natural morphism from $X^i$ to $\Sym^i X$. Let $f$ be the morphism from $\wt{\Sym^m X}$ to $\Sym^m X$. Let $\Gamma'$ be equal to $(id\times f)^*(\Gamma)$ supported on $\Sym^n X\times\wt{\Sym^m X}$. As a first step we prove the following.

\begin{lemma}
\label{lemma1}
The homomorphism $\Gamma'_*\circ \pi_*$ is induced by $(\pi\times id)^*\Gamma'$.
\end{lemma}
\begin{proof}
By definition
$$\Gamma'_*\circ \pi_*(V)=pr_{\wt{\Sym^m X*}}(\pi_*(V)\times \wt{\Sym^m X}.\Gamma')$$
which can be written as
$$pr_{\wt{\Sym^m X*}}((\pi_*\times id_*)(V\times \wt{\Sym^m X}).\Gamma')\;.$$
By the projection formula that is equal to
$$pr_{\wt{\Sym^m X*}}(\pi_*\times id_*)((V\times \wt{\Sym^m X}).(\pi\times id)^*\Gamma')$$
which is nothing but
$$pr_{\wt{\Sym^m X*}}((V\times \wt{\Sym^m X}).(\pi\times id)^*\Gamma')\;.$$
So the homomorphism $\Gamma'_*\circ \pi_*$ is induced by $(\pi\times \id)^*\Gamma'.$
\end{proof}

Now consider the following commutative diagram.
$$
  \diagram
   \wt{Sym^m X}\times \wt{\Sym^m X}\ar[dd]_-{f\times f} \ar[rr]^-{\pi\times id} & & \Sym^n X\times \wt{\Sym^m X} \ar[dd]^-{id\times f} \\ \\
  \Sym^m X\times \Sym^m X \ar[rr]^-{i\times id} & & \Sym^n X\times \Sym^m X
  \enddiagram
  $$
Here $i$ is the closed embedding of $\Sym^{m}X$ into $\Sym^n X$. So we have
$$(id\times f)\circ (\pi\times id)=(i\times id)\circ (f\times f)$$
so we get that
$$(\pi\times id)^*(id\times f)^*\Gamma=(f\times f)^*(i\times id)^*\Gamma$$
but we know by \cite{Collino} that
$$(i\times id)^*\Gamma=\Delta+D$$
where $\Delta $ is the diagonal of $\Sym^m X\times \Sym^m X$ and $D$ is supported on $\Sym^m X\times \Sym^{m-1}X$. Therefore
$$(f\times f)^*(\Delta+D)=(f\times f)^*\Delta+(f\times f)^* D`\;.$$
Now
$$(f\times f)^{-1}(\Delta)=\{(x,y)|f(x)=f(y)\}=\Delta_{\wt{\Sym^m X}}\cup V$$
where $V$ is supported on $E\times E$.
Since $f$ is birational we will have
$$(f\times f)^*(\Delta)=\Delta_{\wt{\Sym^m X}}+dV$$
let
$$(f\times f)^*(D)=D'$$
where $D'$ is supported on $\pi^{-1}(\wt{\Sym^{m} X})\times \wt{\Sym^{m-1} X}$, where $\wt{\Sym^{m-1} X}$ is the strict transform of $\Sym^{m-1} X$ under the blow-up.
Let $\wt{X_0(m)}$ be
$$\wt{\Sym^m X}\setminus (\wt{\Sym^{m-1}X}$$
let $\rho$ denote the embedding of $\wt{X_0(m)}$ into $\wt{\Sym^m X}$.
Then by Chow moving lemma we can take the support of a zero cycle away from $E$ and we have
$$\rho^*(\Gamma'_*\circ \pi_*(z))=\rho^*(pr_{\wt{\Sym^m X*}}(z\times \wt{\Sym^m X}).(\pi\times id)^*\Gamma')$$
that is equal to, by the previous calculation
$$\rho^*(pr_{\wt{\Sym^m X*}}(V\times \wt{\Sym^m X}).(\Delta_{\wt{\Sym^m X}}+E+D'))$$

which is
$$\rho^*(Z+Z_1)$$
where $Z_1$ is supported on the union $\wt{\Sym^{m-1}}X$. Since $\rho^*(Z_1)=0$
we get that
$$\rho^*\Gamma'_*\pi_*=\rho^*\;.$$
Now consider the following commutative diagram.
$$
  \xymatrix{
     \CH_0(\wt{Sym^{m-1}X}) \ar[r]^-{j_{*}} \ar[dd]_-{}
  &   \CH_0(\wt{\Sym^m X}) \ar[r]^-{\rho_0^{*}} \ar[dd]_-{\pi_{*}}
  & \CH_0(X_0(m))  \ar[dd]_-{}  \
  \\ \\
   \CH_0(\wt{\Sym^{m-1}X}) \ar[r]^-{j'_*}
    & \CH_0(\Sym^{n} X) \ar[r]^-{}
  & \CH_0(U)
  }
$$
 $U$ is the complement of $A$ in $\Sym^n X$.
Suppose that
$$\pi_*(z)=0$$
then by the above we get that
$$\rho^*(\Gamma'_*\pi_*(z))=\rho^*(z)=0\;.$$
So by the exactness of the first row we get that there exists $z'$ such that
$$j_*(z')=z\;.$$
by the commutativity of the previous rectangle it follows that
$$j'_{*}(z')=0\;.$$

Now by the assumption we have blown up $\Sym^{m-1}X$, along its intersection with $Z$, so we get that, therefore by the assumption we have that
$$\CH_0(\wt{\Sym^{m-1}X})\to \CH_0(\Sym^n X)$$
is injective. Then it would follow that  $z'=0$ and hence $z=0$. So $\pi_*$ is injective.

\end{proof}

\begin{corollary}
Let $Z$ be a closed point on $\Sym^m X$ which is not in $\Sym^{m-1}X$. Let $\wt{\Sym^{m}X}$ be the blow-up along $Z$. Then $\pi_*$ from $\CH_0(\wt{\Sym^{m}(X)})$  to $\CH_0(\Sym^n X)$ is injective.
\end{corollary}

\begin{proof}
The proof follows from observing the fact that when we blow up along the point, $\Sym^{m-1} X$ remains unchanged. So we have the injectivity of $\CH_0(\wt{\Sym^{m-1}X})\to \CH_0(\Sym^n X)$. Then the proof follows from the previous proposition \ref{prop1}.
\end{proof}


\section{Base change of symmetric powers}
In this section we are going to prove that if we consider a Cartesian square
$$
  \diagram
   {Sym^n X}\times_Z Y\ar[dd]_-{} \ar[rr]^-{} & & Y \ar[dd]^-{} \\ \\
  \Sym^n X \ar[rr]^-{} & & Z
  \enddiagram
  $$
where $Y\to Z$ is an embedding then the inclusion $\Sym^m X\times_Z Y$ to $\Sym^n X\times_Z Y$ induces an injective push-forward homomorphism at the level of Chow groups (such a similar result is proved in \cite{BI} considering the group of algebraic cycles modulo algebraic equivalence. We present the proof here modulo rational equivalence to discuss further implications of it and for further computations).

\begin{proposition}
\label{prop2}
Let $j$ denote the embedding of $\Sym^m X\times_Z Y$ to $\Sym^n X\times_Z Y$ for $m\leq n$. Assume that the fiber products $(\Sym^i X\times_Z Y)$ are smooth for all $i$. Then $j_*$ is injective at the level of Chow groups.
\end{proposition}
\begin{proof}

Let $i$ be the embedding of $\Sym^m X\to \Sym^n X$, where $m\leq n$. Let $j$ denote the embedding of $\Sym^m X\times_Z Y\to \Sym^n X\times_Z Y$. Let $\Gamma$ be as before,
$$\Gamma=\pi_n\times \pi_n(Graph(pr_{n,m}))$$
where $pr_{n,m}$ is the projection from $X^n$ to $X^m$. $\pi_i$ is the natural morphism from $X^i$ to $\Sym^i X$. Let $\pi$ denote the projection morphism from $(\Sym^n X\times_Z Y)\times (\Sym^m X\times_Z Y)\to \Sym^n X\times \Sym^m X$. Then consider the correspondence
$$\pi^*(\Gamma)=\Gamma'$$
supported on $(\Sym^n X\times_Z Y)\times (\Sym^m X\times_Z Y)$. Arguing as in \ref{lemma1} in the previous section we can prove that $\Gamma'_*j_*$ is induced by $(j\times id)^*\Gamma'$, which is equal to $(j\times id)^*\pi^*\Gamma=(\pi\circ (j\times id))^*\Gamma$. Now we have the following commutative diagram.

$$
  \diagram
   {Sym^m X\times_Z Y}\times {\Sym^m X\times_Z Y}\ar[dd]_-{\pi'} \ar[rr]^-{j\times id} & & {\Sym^n X\times_Z Y}\times {\Sym^m X\times_Z Y} \ar[dd]^-{\pi} \\ \\
  \Sym^m X\times \Sym^m X \ar[rr]^-{i\times id} & & \Sym^n X\times \Sym^m X
  \enddiagram
  $$
So we get that
$$\pi\circ (j\times id)=(i\times id)\times \pi'$$
therefore we have that
$$(\pi\circ (j\times id))^*\Gamma=\pi'^*(i\times id)^*\Gamma\;.$$
Now
$$(i\times id)^*\Gamma=\Delta+Y_1$$
where $\Delta$ is the diagonal in $\Sym^m X\times \Sym^m X$ and $Y_1$ is supported on $\Sym^m X\times \Sym^{m-1}X$. Now we compute $\pi'^*(\Delta)$, that is
$$\{([x_1,\cdots,x_m],y)([x_1',\cdots,x_m'],y')|
[x_1,\cdots,x_m]=[x_1',\cdots,x_m']\}$$
but by the definition of fibered product we have that
$$f([x_1,\cdots,x_m])=g(y)=g(y')$$
assuming $g$ to be an embedding we get that $y=y'$. So
$$\pi'^*\Delta=\Delta_{\Sym^m X\times_Z Y}\;.$$
Now
$$\pi'^*(Y)=\{(([x_1,\cdots,x_m],y),([x_1',\cdots,x_m'],y'))\}$$
where $[x_1',\cdots,x_m']=[y_1,\cdots,y_{m-1},p]$, which means that
$$\pi'^* Y$$ is supported on
$$(\Sym^m X\times_Z Y)\times (\Sym^{m-1}X\times_Z Y)\;.$$
So
$$\pi'^*(\Delta+Y_1)=\Delta_{\Sym^m X\times_Z Y}+Y_2$$
where $Y_2$ is supported on
$$(\Sym^m X\times_Z Y)\times (\Sym^{m-1}X\times Y)\;.$$
Now consider
$$\rho:\Sym^m X\times_Z Y\setminus \Sym^{m-1}X\times_Z Y\to \Sym^m X\times_Z Y;,$$
then
$$\rho^*\Gamma'_*j_*(V)=\rho^*(V\times \Sym^m X\times_Z Y.(\Delta_{\Sym^m X\times_Z Y}+Y_2))=\rho^*(V+V_1)=\rho^*(V)\;.$$
Here $V_1$ is supported on $\Sym^{m-1}X\times_Z Y$.

Now  to prove that $j_*$ is injective, we apply induction on $m$. If $m=0$, then since $Y\to Z$ is an embedding we have $\Sym^m \times_Z Y$ is a point. Since $\Sym^n X\times_Z Y$ is projective we have that the inclusion of the point into $\Sym^n X\times_Z Y$ induces injective $j_*$.

Consider the following commutative diagram,

$$
  \xymatrix{
   \CH^*(\Sym^{m-1}X\times_Z Y) \ar[r]^-{j'_{*}} \ar[dd]_-{}
  &   \CH^*({\Sym^m X\times_Z Y}) \ar[r]^-{\rho^{*}} \ar[dd]_-{j_{*}}
  & \CH^*(X_0(m))  \ar[dd]_-{}  \
  \\ \\
   \CH^*(\Sym^{m-1}X\times_Z Y) \ar[r]^-{j''_*}
    & \CH^*(\Sym^{n} X\times_Z Y) \ar[r]^-{}
  & \CH^*(U)
  }
$$
Here $X_0(m)$ is the complement of $\Sym^{m-1}X\times_Z Y$ in $\Sym^m X\times_Z Y$ and $U$ is the complement of $\Sym^{m-1}X\times_Z Y$ in $\Sym^n X\times_Z Y$.

Now suppose that
$$j_*(z)=0$$
that will imply that
$$\rho^*\Gamma'_*j_*(z)=0$$
that is
$$\rho^*(z)=0\;.$$
So by the exactness of the first row we get that
$$j'_*(z')=z\;.$$
Now we have that
$$j_*\circ j'_*(z')=j''_*(z')$$
but by the induction hypothesis we have $j''_*$ is injective from $\Sym^{m-1}X\times_Z Y$ to $\Sym^n X\times_Z Y$. Therefore by the commutativity we have that
$$j''_*(z')=0$$
hence $z'=0$ consequently $z=0$. So $j_*$ is injective.
\end{proof}

\begin{corollary}
Let $Y$ be a closed  subscheme of $\Sym^n X$. Let $i$ denote the closed embedding of $\Sym^m X$ into $\Sym^n X$. Consider $j:Y\cap \Sym^m X\to Y$ and assume that $Y\cap \Sym^i X$ is smooth for all $i$. Then $j_*$ is injective at the level of Chow groups.
\end{corollary}

\begin{proof}
Follows from the previous proposition with $Z=\Sym^n X$.
\end{proof}

Now suppose we consider a projective algebraic variety $Y$ with an involution $i$. Suppose we have the following cartesian square.

$$
  \diagram
   {Sym^n X}\times_{Y/i} Y\ar[dd]_-{} \ar[rr]^-{} & & Y \ar[dd]^-{} \\ \\
  \Sym^n X \ar[rr]^-{} & & Y/i
  \enddiagram
  $$
\begin{proposition}
Let $j$ be the inclusion of $\Sym^m X\times_{Y/i}Y$ into $\Sym^n X_{Y/i}Z$. Assume that these fiber products are smooth. Then $j_*$ is an injection from $\CH_*\left(\frac{\Sym^m X\times_{Y/i}Y}{i}\right)_{\QQ}$ to $\CH_*\left(\frac{\Sym^n X\times_{Y/i}Y}{i}\right)_{\QQ}.$
\end{proposition}

\begin{proof}
The proof goes along the same line as in the previous proposition \ref{prop2}. Only difference will come in the place of
$$\pi'^*(\Delta)$$
that will be the union
$$\Delta_{\Sym^m X\times_{Y/i}Y}\cup i(\Delta_{\Sym^m X\times_{Y/i}Y})$$
where
$$i(\Delta_{\Sym^m X\times_{Y/i}Y})$$
is
$$\{(([x_1,\cdots,x_m],y),([x_1,\cdots,x_m],i(y)))\}\;.$$
Now we compute $i(\Delta_{\Sym^m X\times_{Y/i}Y}).V\times (\Sym^m X\times_{Y/i}Y)\;.$
Set theoretically that is
$$(([x_1,\cdots,x_m],y),([x_1',\cdots,x_m'],y'))$$
such that
$$[x_1,\cdots,x_m]=[x_1',\cdots,x_m']$$
and
$$y=i(y')$$
therefore we denote
$$pr_{\Sym^m X\times_{Y/i}Y}(i(\Delta_{\Sym^m X\times_{Y/i}Y}).V\times \Sym^m X\times_{Y/i}Y)=i(V)\;.$$
Suppose that $V$ is invariant under $i$, that is $i(V)=V$. Then we get that
$$\rho^*\Gamma'_*j_*(V)=\rho^*(V+i(V))=2\rho^*(V)$$
where $\rho$ is the embedding of complement of $\Sym^{m-1}X\times_{Y/i}Y$ into $\Sym^m X\times_{Y/i}Y$. Now consider  the following diagram.
$$
  \xymatrix{
   \CH^*(\Sym^{m-1}X\times_{Y/i} Y)_{\QQ} \ar[r]^-{j'_{*}} \ar[dd]^-{}
  &   \CH^*({\Sym^m X\times_{Y/i} Y})_{\QQ} \ar[r]^-{\rho^{*}} \ar[dd]_-{j_{*}}
  & \CH^*(X_0(m))_{\QQ}  \ar[dd]_-{}  \
  \\ \\
    \CH^*(\Sym^{m-1}X\times_{Y/i} Y)_{\QQ} \ar[r]^-{j''_*}
    & \CH^*(\Sym^{n} X\times_{Y/i} Y)_{\QQ} \ar[r]^-{}
  & \CH^*(U)_{\QQ}
  }
$$
Here $X_0(m),U$ are the complements of $\Sym^{m-1}X\times_{Y/i}Y$ in $\Sym^m X\times_{Y/i}Y,\Sym^n X\times_{Y/i}Y$ respectively.
Then $j_*(z)=0$ implies that
$$\rho^*(2z)=0$$
if $z$ belongs to $\CH_*\left(\frac{\Sym^{m}X\times_{Y/i}Y}{i}\right)_{\QQ}$. Then arguing as in \ref{prop2} we get that $2z=0$, since we are working with $\QQ$ coefficients we get that $z=0$. Hence we have that $j_*$ from $\CH_*\left(\frac{\Sym^m X\times_{Y/i}Y}{i}\right)_{\QQ}$ to $\CH_*\left(\frac{\Sym^n X\times_{Y/i}Y}{i}\right)_{\QQ}$ is injective.

\end{proof}
\begin{corollary}
Let $C$ be a smooth projective curve. Let $J(C)$ denote the Jacobian of $C$. Consider the involution $i:a\mapsto -a$ on $J(C)$. Let $K=J(C)/i$ denote the Kummer variety associated to $J(C)$. Consider the following fiber square.
$$
  \diagram
   {Sym^n X}\times_{K} J(C)\ar[dd]_-{} \ar[rr]^-{} & & J(C) \ar[dd]^-{} \\ \\
  \Sym^n X \ar[rr]^-{} & & J(C)/i=K
  \enddiagram
  $$
Then the natural inclusion morphism $j$ from $\Sym^m X\times_K J(C)$ to $\Sym^n X\times_K J(C)$ induces injective pushforward homomorphism from $\CH_*\left(\frac{\Sym^m X\times_K J(C)}{i}\right)_{\QQ} $ to $\CH_*\left(\frac{\Sym^n X\times_K J(C)}{i}\right)_{\QQ}$.
\end{corollary}
\begin{proof}
We get it by replacing $Y$ in the previous proposition by $J(C)$.
\end{proof}
\begin{theorem}
Let $G$ be a finite group acting on $Y$ and we have a fiber square
$$
  \diagram
   {Sym^n X}\times_{Y/G} Y\ar[dd]_-{} \ar[rr]^-{} & & Y \ar[dd]^-{} \\ \\
  \Sym^n X \ar[rr]^-{} & & Y/G
  \enddiagram
  $$
Suppose that the fiber product $\Sym^i X\times_{Y/G}Y$ is smooth for all $i$. Then the embedding $\Sym^m X\times_{Y/G}Y$ into $\Sym^n X\times_{Y/G}Y$ induces injective push-forward homomorphism from $\CH_*\left(\frac{\Sym^m X\times_{Y/G}Y}{G}\right)_{\QQ}$ into $\CH_*\left(\frac{\Sym^n X\times_{Y/G}Y}{G}\right)_{\QQ}$.
\end{theorem}

\section{Collino's theorem for branched covers of a smooth curve}
Let $C'\to C$ be a finite $i:1$ morphism of smooth projective curves  such that $\ZZ_i$ acts on $C'$. Then we have the induced morphism $\Sym^n C'\to \Sym^n C$. Also we have the embedding $\Sym^m C\to \Sym^n C$. Consider the composite $\Sym^m C'\to \Sym^m C\to \Sym^n C$. In this section we are going to understand the kernel of the  push-forward homomorphism at the level of Chow groups induced by this morphism.

For that consider the correspondence on $C^n\times C'^m$ given by the following Cartesian square.

$$
  \diagram
   C^n\times_{C^m}C'^m\ar[dd]_-{} \ar[rr]^-{} & & C'^m \ar[dd]^-{} \\ \\
  C^n \ar[rr]^-{} & & C^m
  \enddiagram
  $$
Here $C^n\to C^m$ is the projection morphism. That is the correspondence $\Gamma'$ is the pullback of the Graph of the projection. In terms of explicit formula we have
$$\Gamma'=\{((x_1,\cdots,x_n),(x_1',\cdots,x_m'))|
(x_1,\cdots,x_m)=(f(x_1'),\cdots,f(x_m'))\}$$
where $f$ is the given finite morphism from $C'$ to $C$. Then consider $\Gamma$ to be $(\pi_n\times \pi_m')(\Gamma')$ where $\pi_n,\pi_m'$ are the quotient maps to the respective symmetric powers. Let us denote the morphism from $\Sym^m C'$ to $\Sym^n C$ by $j$. Then arguing as in \ref{lemma1} we get that $\Gamma_*\circ j_*$ is induced by the correspondence $(j\times id)^*\Gamma$. Now we compute $(j\times id)^{-1}\Gamma$.

So $(j\times id)^{-1}(\Gamma)$ is nothing but
$$\{([x_1',\cdots,x_m'][y_1',\cdots,y_m'])|(j\times id)([x_1',\cdots,x_m'][y_1',\cdots,y_m'])\in \Gamma\}\;.$$
That would mean
$$(j([x_1',\cdots,x_m']),[y_1',\cdots,y_m'])\in \pi_n\times \pi_m'(\Gamma')\;.$$
That implies
$$([f(x_1'),\cdots,f(x_m'),p,\cdots,p],[y_1',\cdots,y_m'])\in \pi_n\times \pi_m'(\Gamma')$$
which means that either
$$f(x_i')=f(y_i')$$
for all $i$ all for some $i$ we have $f(y_i')=p$. Supposing that $f$ is $i:1$ we get that
$$(j\times id)^*(\Gamma)=\sum_i d_i\Delta_i+D$$
where $\Delta_i$ is the image of the diagonal on $\Sym^m C'\times \Sym^m C'$ under the action of the group $\ZZ_i$ and $D$ is supported on $\Sym^m C'\times \Sym^{m-1}C'$. Then
$$\Gamma_*j_*(Z)=pr_{\Sym^m C'*}(Z.(j\times id)^*(\Gamma))$$
is equal to $\sum_i d_i Z_i+Y$
where $Z_i$ is nothing but
$$pr_{\Sym^m C'*}(Z.\Delta_i)\;,$$
and $Y$ is supported on $\Sym^{m-1}C'$.
Then let $\rho $ be the inclusion of the complement of $\Sym^{m-1}C'$. Then we get that
$$\rho^*\Gamma_*j_*(Z)=\sum_i d_i\rho^*(Z_i)\;.$$
Now consider the following commutative diagram.
$$
  \xymatrix{
   \CH_*(\Sym^{m-1}C') \ar[r]^-{j'_{*}} \ar[dd]^-{}
  &   \CH_*(\Sym^{m}C') \ar[r]^-{\rho^{*}} \ar[dd]_-{j_{*}}
  & \CH_*(X_0(m))  \ar[dd]_-{}  \
  \\ \\
    \CH_*(\Sym^{m-1}C') \ar[r]^-{j''_*}
    & \CH_*(\Sym^{n} C) \ar[r]^-{}
  & \CH^*(U)
  }
$$
Here $X_0(m),U$ are complements of $\Sym^{m-1}C'$ or its image in $\Sym^m C'$ and $\Sym^n C$ respectively. Suppose that $j_*(z)=0$, then by the previous computation it follows that
$$\rho^*\Gamma_*j_*(z)=0$$
that implies  $\rho^*(\sum_i d_i z_i)=0$. Consider $z$ to be in $\CH_*(\Sym^{m-1} C')^{\ZZ_i}$. That gives us that
$$\rho^*(dz)=0$$
so that would mean that $dz=j'_*(z')$. Now assume that by induction we have the kernel of $\CH_*(\Sym^{m-1} C')^{\ZZ_i}\to \CH_*(\Sym^n C)$ is torsion. Then that would mean that $d'z'=0$. That will give us
$dd'z=0$. So we will get that the following.

\begin{theorem}
Elements of the kernel of the push-forward from $\CH_*(\Sym^m C')^{\ZZ_i}$ to $\CH_*(\Sym^n C)$ are torsion.
\end{theorem}

\section{Closed subscheme inside symmetric powers of a curve}
Let $E$ be a non-singular closed  subscheme inside $\Sym^m C$ such that all its intersection with $\Sym^l X$ for $l\leq m$ are smooth, then we consider the embedding of $\Sym^m C$ into $\Sym^n C$. We want to prove that $j:E\to \Sym^n C$ gives rise to an injective push-forward homomorphism at the level of Chow groups. Consider the projection from $C^n$ to $C^m$. Consider the correspondence given by the fiber product of $C^n$ and $\pi^{-1}_m(E)$ over $C^m$. Call this  correspondence $\Gamma'$ . Then define $\Gamma$ to be $\pi_n\times \pi_m(\Gamma')$, that will give us a correspondence on $\Sym^n C\times E$. Then by \ref{lemma1} we get that the homomorphism $\Gamma_*j_*$ is induced by the cycle $(j\times id)^*(\Gamma)$. Now we compute this cycle.
So $(j\times id)^{-1}(\Gamma)$ is nothing but
$$\{([e_1,\cdots,e_m],[e_1',\cdots,e_m'])|
([e_1,\cdots,e_m,p\cdots,p],[e_1',\cdots,e_m'])\in \Gamma\}\;.$$
That would mean the following
$(e_1,\cdots,p,\cdots,e_m)$ and $(e_1,\cdots,e_m,p,\cdots,p)$ are in $\pi_n^{-1}(E)$ and
$(e_1',\cdots,e_m')$
is in the image of the projection. So we have
$$(e_1',\cdots,e_m')=(e_1,\cdots,e_m)$$
or
$$e_i'=p$$
for some $i$. That would mean that $(j\times id)^{-1}(\Gamma)=\Delta_E\cup Y$ where $Y$ is supported on $\Sym^{m-1}C\cap E$. Arguing as in \cite{Collino} we get that
$$(j\times id)^*(\Gamma)=d\Delta_E+D$$
where $D$ is supported on $\Sym^{m-1}C\cap E$. The chow moving lemma holds by the assumption that $E$ is non-singular and its intersections with all $\Sym^l X$ are non-singular for all $l\leq m$.  Then consider $\rho$ to be the open immersion of the complement of $\Sym^{m-1}C\cap E$ in $E$. Since $D$ is supported on $\Sym^{m-1}C\cap E$ we get that
$$\rho^*\Gamma_*j_*(Z)=\rho^*(dZ)\;.$$
As previous consider the diagram.
$$
  \xymatrix{
   \CH_*(\Sym^{m-1}C\cap E) \ar[r]^-{j'_{*}} \ar[dd]^-{}
  &   \CH_*(E) \ar[r]^-{\rho^{*}} \ar[dd]_-{j_{*}}
  & \CH_*(X_0(m))  \ar[dd]_-{}  \
  \\ \\
    \CH_*(\Sym^{m-1}C\cap E) \ar[r]^-{j''_*}
    & \CH_*(\Sym^{n} C) \ar[r]^-{}
  & \CH^*(U)
  }
$$

$X_0(m),U$ are complement of $\Sym^{m-1}C\cap E$ in $E,\Sym^n C$ respectively. Then suppose that we have
$$j_*(z)=0$$
that gives us that
$$\rho^*\Gamma_*j_*(z)=\rho^*(dz)=0$$
so there exists some $z'$ such that $j_*'(z')=dz$. But by the above diagram we have $j''_*(z')=0$. So by induction if we assume that $j''_*$ has torsion kernel then we get that $d'z'=0$, so we have $dd'z=0$. So the kernel of the map from $\CH_*(E)\to \CH_*(\Sym^n C)$ is torsion, consequently $\CH_*(E)\to \CH_*(\Sym^m C)$ has torsion kernel.

\begin{theorem}
Let $E$ be a non-singular closed subscheme in $\Sym^m C$ such that its intersections with all $\Sym^l X$ for $l\leq m$ are also non-singular. Then the closed embedding of $E$ into $\Sym^n C$ for $m\leq n$ induces a push-forward homomorphism at the level of Chow groups which has torsion kernel.
\end{theorem}

\end{document}